\documentclass{amsart}
\usepackage{amscd,amssymb,amsmath,graphicx,verbatim}
\usepackage[dvips]{hyperref}
\usepackage[TS1,OT1,T1]{fontenc}
\newcommand{\Gal}{\operatorname{Gal}}
\newcommand{\AGL}{\operatorname{AGL}}
\newcommand{\PSL}{\operatorname{PSL}}

\newcommand{\supp}{\operatorname{supp}}
\newcommand{\Tr}{\operatorname{Tr}}
\newcommand{\QQ}{{\mathbb Q}}
\newcommand{\CC}{{\mathbb C}}
\newcommand{\FF}{{\mathbb F}}
\newcommand{\ZZ}{{\mathbb Z}}
\newcommand{\eee}{\hfill$\Box$}
\newtheorem{theorem}{Theorem}
\newtheorem{corollary}{Corollary}
\newtheorem{lemma}{Lemma}
\newtheorem{remark}{Remark}
\begin{document}
\title{On Galois Groups of Prime Degree Polynomials with Complex Roots}
\author{Oz Ben-Shimol}
\maketitle
%
%
\begin{abstract}
    Let $f$ be an irreducible polynomial of prime degree $p\geq 5$ over $\QQ$,
    with precisely $k$ pairs of complex roots.
    Using a result of Jens H\"{o}chsmann (1999),
    we show that if $p\geq 4k+1$ then $\Gal(f/\QQ)$ is isomorphic to $A_{p}$ or $S_{p}$.
    This improves the algorithm for computing the Galois group of an irreducible polynomial of prime degree,
    introduced by A. Bialostocki and T. Shaska.
    \\ \indent
    If such a polynomial $f$ is solvable by radicals then its Galois group is a Frobenius group of degree p.
    Conversely, any Frobenius group of degree p and of even order, can be
    realized as the Galois group of an irreducible polynomial of degree $p$ over $\QQ$ having complex roots.
\end{abstract}
\maketitle
%
%
\section{Introduction}
A classical theorem in Galois theory says that  an irreducible
polynomial $f$ of prime degree $p\geq 5$ over $\QQ$ which has
precisely one pair of complex (i.e., non-real) roots, has the
symmetric group $S_{p}$ as its Galois group over $\QQ$ (see e.g.,
Stewart[18]). It is natural then to ask the following question: let
$k$ be a positive integer and $f$ an irreducible polynomial of prime
degree $p$ with precisely $k$ pairs of complex roots. What is its
Galois group $\Gal(f/\QQ)$?. If one tries to imitate the proof of
the classical theorem (i.e., the case $k=1$), one would find,
constructively, the subgroup of $S_{p}$ which is generated by the
$p$-cycle $(1 \ 2 \ ... \ p)$ and an involution $(a_1 \
a_2)\cdots(a_{2k-1} \ a_{2k})$. My unsuccessful attempts (so far) to
solve the problem in this way indicated that the difference between
the degree $p$ and the number $2k$ of the complex roots, need not be
"large" in order to obtain the alternating group $A_{p}$ at least
(i.e., $\Gal(f/\QQ)$ is isomorphic to $A_{p}$ or $S_{p}$).
\\ \indent
More general observations on such permutation groups brings us to a
well-known problem in the theory of permutation groups: let $G$ be a
$2$-transitive permutation group of degree $n$ which does not
contain the alternating group $A_{n}$, and let $m$ be its minimal
degree. Find the infimum for $m$ in terms of $n$.
\\ \indent
If $f$ is an irreducible polynomial of prime degree $p$ with $k>0$
pairs of complex roots, where $p>2k+1$, then its Galois group
$\Gal(f/\QQ)$ is $2$-transitive of degree $p$, with minimal degree
at most $2k$. Therefore, if $B(p)$ is a lower bound for the minimal
degree, then $\Gal(f/\QQ)$ necessarily contains the alternating
group $A_{p}$ when $2k\leq B(p)$. Thus, as $B(p)$ approaches the
infimum, the difference $p-2k$ gets smaller, as required.
\\ \indent
Returning to the group-theoretic problem stated above (for degree
$n$, not necessarily a prime), Jordan [10] showed that
$B(n)=\sqrt{n-1}+1$ is a lower bound for the minimal degree. A
substantial improvement of this bound is due to Bochert [3] who
showed that $B(n)=n/8$, and if $n>216$ then one has an even better
bound, namely $B(n)=n/4$. Proofs for the Jordan and Bochert
estimates can be found also in Dixon \& Mortimor [7], Theorem 3.3D
and Theorem 5.4A, respectively. More recently, Liebeck and Saxl
[11], using the classification of finite simple groups, have
proved $B(n)=n/3$.

Finally, H\"{o}chsmann [8], using a concept suggested by W.Knapp
which refines the notion of minimal degree in a natural way, namely,
\emph{$r$-minimal degree} $m_r(G)$, where $r$ is a prime divisor of
the order of the group $G$, gave some better estimates, which in the
worst case meet Liebeck and Saxl's bounds. Since the group we are
dealing with is of prime degree, and we have information about its
$2$-minimal degree, Hochmann's result serves us better than that of
Liebeck and Saxl.
\\ \indent
The paper of A.Bialostocki and T.Shaska [2] focuses on the practical
aspects of this theoretical problem, in the process of computing the
Galois group of prime degree polynomials over $\QQ$: 1. The existing
techniques, which are mainly based on a theorem of Dedekind (see Cox
[6, Theorem 13.4.5]), are expensive and many primes $p$ might be
needed in the process. 2. Polynomials in general have plenty of
complex roots. 3. Checking whether a polynomial has complex roots is
very efficient since numerical methods can be used. Therefore,
checking first if the polynomial has complex roots, and then use a
"good" bound for the difference between the polynomial's degree and
the number of its complex roots, makes the computation of its Galois
group much easier. However, they make a use of estimate due to
Jordan (summarized in Wielandt [19, page 42]), which is not sharp at
all (as the authors point out in their paper). In fact, Jordan's
bound holds for any primitive group of any finite degree - not
necessarily $2$-transitive of prime degree. In the present paper, we
improve their algorithm and discuss some theoretical aspects of the
subject.
%
%
%
\section{Galois groups of prime degree polynomials with complex roots}
A \emph{Frobenius group} is a transitive permutation group which
is not regular, but in which only the identity has more then one
fixed point. In other words, a Frobenius group $G$ is a transitive
permutation group on a set $\Omega$ in which $G_{\alpha}\neq 1$
for some $\alpha\in\Omega$, but $G_{\alpha}\cap G_{\beta}=1$ for
all $\alpha,\beta\in\Omega$, $\alpha\neq\beta$. It can be shown
that the set of elements fixing no letters of $\Omega$, together
with the identity, form a normal subgroup $K$ called the
\emph{Frobenius kernel} of $G$. Frobenius groups are characterized
as non-trivial semi-direct products $G=K\rtimes H$ such that no
element of $H\setminus\{1\}$ commutes with any element of
$K\setminus\{1\}$. Basic examples of Frobenius groups are the
subgroups of $\AGL_{1}(F)$ - the group of the $1$-dimensional
affine transformations of a field $F$, i.e. the group consisting
of the permutations of the form
$t_{\alpha,\beta}:\zeta\mapsto\alpha\zeta+\beta$, $\alpha\in
F^{*}$, $\beta,\zeta\in F$. Clearly, $\AGL_{1}(F)\cong F\rtimes
U$, where $U$ is a non-trivial subgroup of $F^{*}$. Identifying
$U$ with $\{0\}\rtimes U$, it is easy to verify that no nontrivial
subgroup of $U$ is normal in $\AGL_{1}(F)$. In particular, if
$F=\FF_{p}$ - the field of $p$ elements ($p$ prime), then
$\AGL_{1}(p):=\AGL_{1}(\FF_{p})\cong\FF_{p}\rtimes U$, where $U$
is a subgroup of $\FF_{p}^{*}$ (so $U$ is a cyclic of order $n$,
where $n\neq 1$ and $n$ divides $p-1$), is a Frobenius group of
degree $p$. The structure of a Frobenius group of degree $p\geq 5$
is described in the following theorem.
%
%
\begin{theorem} \emph{(Galois)}
Let G be a transitive permutation group of prime degree $p\geq 5$, and of order $>p$.
Then the following statements are equivalent:
\\
\textbf{i}. \ $G$ has a unique $p$-Sylow subgroup. \\
\textbf{ii}. \ $G$ is a solvable group. \\
\textbf{iii}. \ $G$ is isomorphic to a subgroup of $\AGL_{1}(p)$. \\
\textbf{iv}. \ G is a Frobenius group of degree $p$.
\end{theorem}
%
%
\begin{proof}
    See Huppert [9].
\end{proof}
Let $G\cong\FF_{p}\rtimes U$, $U$  cyclic of order $n$, $n\neq 1$,
$n|p-1$, be a Frobenius group of degree $p$. Then it is customary
to denote $G=F_{pn}$. For example, the dihedral group
$D_{2p}=F_{p\cdot 2}$ is a Frobenius group of degree $p$. The
Frobenius groups $F_{p(p-1)}$ appear as Galois groups of the
polynomials $X^{p}-a\in\QQ[X]$, where
$a\in\QQ^{*}\setminus(\QQ^{*})^{p}$. For constructive realization
of Frobenius groups of degree $p$, see A.A.Bruen, C.Jensen and
N.Yui [4].
\\
\indent If $f$ is an irreducible polynomial of degree $p\geq 5$ over
$\QQ$, then its Galois group $G=\Gal(f/\QQ)$, as a permutation group
acting on the $p$-set consisting of the $p$ roots of $f$, is a
transitive group of order $p$ (if and only if $G$ contains a
$p$-cycle). Complex conjugation is a $\QQ$-automorphism of $\CC$
and, therefore, induces a $\QQ$-automorphism of the splitting field
of $f$. This leaves the real roots of $f$ fixed, while transposing
the complex roots. Therefore, if $f$ has a pair of complex roots,
then $|G|>p$. Furthermore, if, in addition, f has more then one real
root, then the complex conjugation has more then one fixed point. In
particular, $G$ is not a Frobenius group of degree $p$. By Theorem
1, $G$ is not solvable, thus, $f$ is not solvable by radicals. So we
have
%
%
\begin{corollary}
    Let $f$ be an irreducible polynomial of prime degree $p\geq 5$ over $\QQ$, which has a pair of complex roots.
    If $f$ is solvable by radicals then $\Gal(f/\QQ)$ is a Frobenius group of degree $p$,
    and $f$ has exactly one real root.
\eee
\end{corollary}
Let $f$ be an irreducible polynomial of prime degree $p\geq 5$ and
with $k>0$ pairs of complex roots. By Corollary 1, if $p>2k+1$ then
$G=\Gal(f/\QQ)$ is not solvable. Our purpose is to show that if
$p\geq 4k+1$ then $G$ contains the alternating group (i.e., $G$
isomorphic to $A_{p}$ or to $S_{p}$).
%
%
\begin{theorem}\emph{(Burnside)}
    A non-solvable transitive permutation group of prime degree is $2$-transitive.
\end{theorem}
Therefore, a transitive permutation group of prime degree is either $2$-transitive or a Frobenius group (see Theorem 1).
\begin{proof}
    See, Burnside [5], or Dixon \& Mortimor [7, Corollary 3.5B].
\end{proof}
Recall that the \emph{minimal degree} $m(G)$ of a permutation group
$G$ acting on a set $\Omega$ is the minimum of the supports of the
non-identity elements: \ $m(G):=\min\{|\supp(x)| : x\in G, x\neq
1\}$. Hence, G is a Frobenius group if and only if it is a
transitive permutation group with minimal degree $|\Omega|-1$, and
by Theorem 1, a transitive permutation group of prime degree $p\geq
5$ and of order $>p$ is not solvable if and only if it has minimal
degree $<p-1$.
\\ \indent
Now, for every prime divisor $r$ of $|G|$ we define the
\emph{minimal r-degree} $m_r(G)$ of $G$ to be the minimum of the
supports of the non-identity r-elements (that is, the non-identity
elements whose order is a power of $r$). Using elementary properties
of the minimal $r$-degrees and together with results based on the
classification of the finite simple groups, J. H\"{o}chsmann [8] has
proved
%
%
\begin{theorem}\emph{(H\"{o}chsmann)}
    Let $G$ be a $2$-transitive group of degree $n$ which does not contain the alternating group, and let $r$ be a
    prime divisor of $|G|$. Then \\
    \textbf{i.}  $m_r(G)\geq\frac{r-1}{r}\cdot n$ \ or \\
    \textbf{ii.} $G\geq\PSL(2,2^m)$, $r=2^{m}-1\geq 7$ is a Mersenne prime and $m_{r}(G)=r=n-2$ or \\
    \textbf{iii.} $G=PSp(2m,2)$, $n=2^{m-1}\cdot(2^{m}-1)$ with $m>2$, $r=2$ and \\
    $m_{r}(G)=\frac{2^{m-1}-1}{2^{m}-1}\cdot n\geq \frac{3}{7}\cdot n$. \\
    In any case $m_{r}(G)\geq\frac{r-1}{r+1}\cdot n$.
\end{theorem}
An immediate consequence (in fact, a special case) of this theorem
is
%
%
\begin{corollary}
    Let $G$ be a $2$-transitive group of prime degree $p$ which does not contain the alternating group. Then
    $m_2(G)\geq\frac{p}{2}$. \eee
\end{corollary}
%
%
\begin{theorem}
    Let $f$ be an irreducible polynomial of prime degree $p\geq 5$ over $\QQ$.
    Suppose that $f$ has precisely $k>0$ pairs of complex roots.
    If $p\geq 4k+1$ then
    $G=\Gal(f/\QQ)$ is isomorphic to $A_{p}$ or to $S_{p}$. Clearly, if $k$ is odd then $G\cong S_{p}$.
\end{theorem}
\begin{proof}
    Complex conjugation has support $2k$, hence $m_2(G)\leq 2k$. By Corollary 1, $G$ is not solvable ($f$ has more than
    one real root). By Theorem 3, $G$ is $2$-transitive and, by Corollary 2, $G$ necessarily contains the alternating group.
\end{proof}
Therefore, the algorithm given in [2] for computing the Galois group
of an irreducible prime degree polynomial, can be improved:
\\
\\
%
%
\textbf{Input:} An irreducible polynomial $f(x)\in\QQ[x]$ of prime degree $p$. \\
\textbf{Output:} The Galois group $\Gal(f/\QQ)$.
\\
begin \\
r:=NumberOfRealRoots(f(x)); \\
k:=(p-r)/2; \\
if $k>0$ and $p\geq 4k+1$ then \\ \indent
   if k is odd then \\ \indent
\ \ \ $\Gal(f/\QQ)=S_{p}$; \\ \indent
   else \\ \indent
\ \ \ if $\Delta(f)$ is a complete square then \\ \indent
\ \ \ \ \ \ $\Gal(f/\QQ)=A_{p}$; \\ \indent
\ \ \ else \\ \indent
\ \ \ \ \ \ $\Gal(f/\QQ)=S_{p}$; \\ \indent
\ \ \  endif; \\ \indent
endif; \\
else  \\ \indent
\ \ ReductionMethod(f(x)); \\
endif \\
end;
%
%
\begin{remark}
    \emph{
    $\Delta(f)$ denotes the discriminant of $f(x)$. It is well known that if $f$ is a polynomial of degree $n$
    with coefficients in a field $K$, char($K$)$\neq 2$, then
    $\Delta(f)$ is a perfect square in $K$ if and only if $\Gal(f/K)$ is
    isomorphic to a subgroup of $A_{n}$. See e.g., Stewart [18, Theorem 22.7].}
\end{remark}
%
%
\begin{remark}
    \emph{
    A short discussion on the reduction modulo $p$ method, can be found in [2] and in Cox [6, page 401].}
\end{remark}
%
%
\begin{remark}
    \emph{
    Corollary 1 in [2] can also be improved: (replace their $r$ with our $k$ - the number of pairs of the complex roots
    of a given irreducible polynomial of prime degree $p$). (i) $k=2$ and $p>7$.
    (ii) $k=3$ and $p>11$. (iii) $k=4$ and $p>13$. (iv) $k=5$ and $p>19$. }
\end{remark}
%
%

\section{ non-real realization of $F_{pn}$}
As stated in Corollary 1, an irreducible solvable polynomial of
prime degree $p\geq 5$ over $\QQ$, which has complex roots, has a
Frobenius group of degree $p$ (and of even order, of course) as its
Galois group over $\QQ$. We shall prove that the related "inverse
problem" has a positive answer - any Frobenius group of degree $p$
and of even order appears as Galois group of an irreducible
polynomial of degree $p$ over $\QQ$ having complex roots.
\begin{theorem}\emph{(Dirichlet)}
    Let $k$,$h$ be integers such that $k>0$ and $(h,k)=1$. Then there are infinitely many primes in the arithmetic
    progression $nk+h$, $n=0,1,2,\ldots$.
\end{theorem}
\begin{proof}
    See e.g., Serre [15] or Apostol [1].
\end{proof}
\begin{lemma}
    Let $l$ be a positive integer, and let $\zeta$ be a primitive $l$-th root of unity.
    Then $1,\zeta,\ldots,\zeta^{\varphi(l)-1}$ form a $\ZZ$-basis for the ring of integers of
    $\QQ(\zeta)$.
\end{lemma}
\begin{proof}
    See e.g., Neukirch [12, Chapter I, Proposition 10.2].
\end{proof}
\begin{lemma}\emph{(Galois)}
    Let $f$ be a polynomial of prime degree over $\QQ$.
    Then, $f$ is solvable by radicals if and only if any two distinct roots of f generate its splitting field.
\end{lemma}
\begin{proof}
    See Cox [6, Theorem 14.1.1].
\end{proof}
\begin{theorem}\emph{(Scholz)}
    A splitting embedding problem has a proper solution over number fields. (That is, let $K$ be a number field and let
    $M/K$ be a Galois extension with Galois group $H$. Suppose that $H$ acts on an abelian group $A$. Then, there exist
    a Galois extension $L/K$ which contains $M/K$ such that $\Gal(L/K)\cong A\rtimes H$).
\end{theorem}
\begin{proof}
    See Scholz [14].
\end{proof}
\begin{theorem}
    Let $F_{pn}$ be a Frobenius group of degree $p$ and of even order.
    Then $F_{pn}$ occurs as Galois group of an irreducible polynomial $f$ of degree $p$ over
    $\QQ$ having complex roots.
    Furthermore, the splitting field of $f$ is $\QQ(a,\mathbf{i}b)$ for every complex root $a+\mathbf{i}b$ of $f$.
\end{theorem}
\begin{proof}
    By Theorem 5, there exist a prime $q$ such that $q\equiv 1(\mod{n})$ and $(q-1)/n$ is odd number.
    Indeed, for every natural number $k$, write $1+(2k-1)n=(1-n)+(2n)k$. So, $(1-n,2n)=1$ since $n$ is even. Thus, such a
    prime $q$ does exist. Let $m$ be a primitive root modulo $q$ (that is, a generator of $\FF_{q}^{*}$).
    Consider the sum
    \begin{equation}
        \alpha_{n}=\zeta_{q}+\zeta_{q}^{m^{n}}+\zeta_{q}^{m^{2n}}+\ldots+\zeta_{q}^{m^{\left(\frac{q-1}{n}-1\right)n}},
    \end{equation}
    where $\zeta_{q}$ is a primitive $q$-th root of unity. Then $\Gal(\QQ(\zeta_{q})/\QQ)$ is cyclic of order $q-1$
    and generated by the automorphism $\sigma:\zeta_{q}\mapsto\zeta_{q}^{m}$. We shall see that $\QQ(\alpha_{n})/\QQ$
    is a non-real $C_{n}$-extension, and then we shall apply Theorem 6. \\
    \textbf{$\mathbf{\QQ(\alpha_{n})/\QQ}$ is a $\mathbf{C_{n}}$-extension:} By the Fundamental Theorem of Galois Theory,
    it is enough to prove $\QQ(\alpha_{n})=\QQ(\zeta_{q})^{\sigma^{n}}$.
    The inclusion $\QQ(\alpha_{n})\subseteq\QQ(\zeta_{q})^{\sigma^{n}}$ is because $\sigma^{n}$ moves cyclicly the
    summands of (1) (in fact, $\alpha_{n}$ is the image of $\zeta_{q}$ under the trace map
    $\Tr_{\QQ(\zeta_{q})/\QQ(\zeta_{q})^{\sigma^n}}$, hence $\alpha_{n}$ is an element of $\QQ(\zeta_{q})^{\sigma^n}$).
    Suppose that $\QQ(\alpha_{n})\subsetneqq\QQ(\zeta_{n})^{\sigma_{n}}$. There exist a proper divisor $d$ of $n$
    such that $\QQ(\alpha_{n})=\QQ(\zeta_{q})^{\sigma^{d}}$. In particular, $\sigma^{d}(\alpha_{n})=\alpha_{n}$, or
    \begin{equation}
        \sum_{j=0}^{\frac{q-1}{n}-1}\zeta_{q}^{m^{jn+d}}-
        \sum_{j=0}^{\frac{q-1}{n}-1}\zeta_{q}^{m^{jn}}=0.
    \end{equation}
    We shall see in a moment that the summands in (2) are distinct in pairs. Taking it as a fact, there are
    $2(q-1)/n$ ($\leq q-1$) summands, and dividing each of them by $\zeta_{q}$
    gives us a linear dependence among the $1,\zeta_{q},\zeta_{q}^2,\ldots,\zeta_{q}^{q-2}$ in contradiction to Lemma 1.
    Now, if $\zeta_{q}^{m^{jn+d}}=\zeta_{q}^{m^{in}}$ for some $i,j=0,1,\ldots,\frac{q-1}{n}-1$, $j\geq i$, then
    $m^{(j-i)n+d}\equiv 1(\mod q)$. $m$ is primitive modulo $q$ so $q-1$ divides $(j-i)n+d$.
    But, $(j-i)n+d<(\frac{q-1}{n}-1)n+n=q-1$, a contradiction. Therefore, all the summands in (2) are distinct in pairs.
    \\
    \textbf{$\mathbf{\alpha_{n}}$ is not real:} No summand in (1) is a complex conjugate of the other.
    Indeed, if
    $\zeta_{q}^{m^{jn}}=\zeta_{q}^{-m^{in}}$ for some $i,j=0,1,\ldots,\frac{q-1}{n}-1$, $j\geq i$, then
    $m^{(j-i)n}\equiv -1(\mod q)$, so $m^{2(j-i)n}\equiv 1(\mod q)$.
    Therefore, the odd number $(q-1)/n$ divides $2(j-i)$, thus divides $j-i$. But $j-i<(q-1)/n$. We conclude that
    no summand in (1) is a complex conjugate of the other.
    Finally, if $\alpha_{n}$ was real, then $\frac{1}{\zeta_{q}}(\alpha_{n}-\overline{\alpha_{n}})=0$ and by the same
    considerations above, we get a contradiction to Lemma 1.
    \\ \indent
    Now by Theorem 6, we can embed the non-real $C_n$-extension $\QQ(\alpha_n)/\QQ$
    in a $F_{pn}$-extension $L/\QQ$ (say). Let $\QQ(\beta)/\QQ$ be an intermediate extension of degree $p$ which
    corresponds to (the isomorphic copy of) $U\cong C_{n}$. No non-trivial subgroup of $U$ is normal in $F_{pn}$, hence
    $L/\QQ$ is the splitting field of the minimal polynomial $f$ of the primitive element $\beta$.
    $f$ is the required polynomial.
    \\ \indent
    If $a+\mathbf{i}b$ is a complex root of $f$ then
    $L=\QQ(a+\mathbf{i}b,a-\mathbf{i}b)=\QQ(a,\mathbf{i}b)$ by Lemma 2 and Corollary 1.
\end{proof}
\begin{remark}
\emph{ Any Frobenius group can be realized as Galois group over
$\QQ$ (the realizations are not necessarily non-real).
I.R.\v{S}afarevi\v{c} [13] proved that any solvable group appears as
Galois group over number fields, and J.Sonn [16,17] proved that any
non-solvable Frobenius group appears as Galois group over $\QQ$.}
\end{remark}
\section{Acknowledgment}
The author is grateful to Moshe Roitman, Jack Sonn, Tanush Shaska
and John Dixon for useful discussions.

%
%
\vskip 1cm
\small
\noindent
Oz Ben-Shimol \\
Department of Mathematics \\
University of Haifa \\
Mount Carmel 31905, Haifa, Israel \\
E-mail Address: obenshim@math.haifa.ac.il
\end{document}